\documentclass[12pt]{article}
\usepackage{amssymb,amsmath,amsthm,epsfig}
\usepackage[margin=25mm]{geometry}
\usepackage{float}
\usepackage{sectsty}
\usepackage{lipsum}
\usepackage[labelfont=bf]{caption}
\allsectionsfont{\centering\MakeUppercase}
\newtheorem{thm}{Theorem}

\renewenvironment{proof}{{\bfseries Proof.}}{}

\makeatletter
\renewcommand{\maketitle}{\bgroup\setlength{\parindent}{0pt}
\begin{flushleft}
  \textbf{\@title}

  \@author
\end{flushleft}\egroup
}
\makeatother

\begin{document}
\title{\bf \Large A note on the minimum reduced reciprocal Randi\'{c} index of $n$-vertex unicyclic graphs}

\author{AKBAR ALI$^{\dag,\ddag,*}$ \& AKHLAQ AHMAD BHATTI$^{\dag}$\\
$^{\dag}$Department of Sciences \& Humanities, National University of Computer \& Emerging Sciences, B-Block, Faisal Town, Lahore-Pakistan.\\
$^{\ddag}$Department of Mathematics, University of Gujrat, Gujrat-Pakistan.\\
$^{*}$Corresponding author: akbarali.maths@gmail.com}

\maketitle

\renewcommand{\abstractname}{ABSTRACT}
\begin{abstract}
\noindent The graph having the minimum reduced reciprocal Randi\'{c} index is characterized among the class of all unicyclic graphs with fixed number of vertices.
\end{abstract}

\section*{Introduction}

All the graphs considered in the present study are simple, finite, undirected and connected.  The vertex set and edge set of a graph $G$ will be denoted by $V(G)$ and $E(G)$ respectively. The degree of a vertex $u\in V(G)$ and the edge connecting the vertices $u$ and $v$ will be denoted by $d_{u}$ and $uv$ respectively. Undefined notations and terminologies from (chemical) graph theory can be found in (Harary, 1969; Trinajsti\'{c}, 1992).

Topological indices are numerical parameters of a graph which are invariant under graph isomorphisms. Randi\'{c} (1975) proposed the following topological index:
\[R(G)=\displaystyle\sum_{uv\in E(G)}(d_{u}d_{v})^{-\frac{1}{2}}\]
for measuring the extent of branching of the carbon-atom skeleton of saturated hydrocarbons and he named it as \textit{branching index}. Nowadays, this topological index is also known as \textit{connectivity index} and the \textit{Randi\'{c} index}. According to Gutman (2013), ``the Randi\'{c} index is the most investigated, most often applied, and most popular among all topological indices. Hundreds of papers and a few books are devoted to this topological index''.

On the other hand, many physico-chemical properties of chemical structures are dependent on the factors different from branching. In order to take these factors into account, Estrada \textit{et al.} (1998) introduced a modified version of the Randi\'{c} index and called it as \textit{atom-bond connectivity (ABC)} index. This index is defined as:
\[ABC(G)=\sum_{uv\in E(G)}\sqrt{\frac{d_{u}+d_{v}-2}{d_{u}d_{v}}}.\]
Details about the chemical applicability and mathematical properties of this index can be found in the survey (Gutman, 2013), recent papers (Ahmadi \textit{et al.}, 2014; Dimitrov, 2014; Goubko \textit{et al.}, 2015; Palcios, 2014; Raza \textit{et al.}, 2015) and the references cited therein.

Inspired by work on the $ABC$ index Furtula, Graovac \& Vuki$\check{c}$evi$\acute{c}$ (2010) gave the following modified version of the $ABC$ index (and hence a modified version of Randi\'{c} index) under the name \textit{augmented Zagreb index} ($AZI$):
\[AZI(G)=\displaystyle\sum_{uv\in E(G)}\left(\frac{d_{u}d_{v}}{d_{u}+d_{v}-2}\right)^{3}.\]
The prediction power of $AZI$ is better than $ABC$ index in the study of heat of formation for heptanes and octanes (Furtula \textit{et al.}, 2010). Details about this index can be found in the survey (Gutman, 2013), recent papers (Ali, Bhatti \& Raza, 2016; Ali, Raza \& Bhatti, 2016; Huang \& Liu, 2015; Zhan \textit{et al.}, 2015) and the references cited therein.

In (Manso \textit{et al.}, 2012), a new topological index (namely \textit{Fi} index) was proposed to predict the normal boiling point temperatures of hydrocarbons. In the mathematical definition of \textit{Fi} index two terms are present. Gutman, Furtula \& Elphick (2014), recently considered one of these terms which is given below:
\[RRR(G)=\sum_{uv\in E(G)}\sqrt{(d_{u}-1)(d_{v}-1)},\]
and they named it as \textit{reduced reciprocal Randi\'{c}} ($RRR$) \textit{index}. In the current study, we are concerned with this recently introduced modified version of the Randi\'{c} index.
In order to get some preliminary information on whether this index possess any potential applicability in chemistry (especially in QSPR/QSAR
studies) Gutman, Furtula \& Elphick (2014) tested the correlating ability of several well known degree based topological indices (along with $RRR$ index) for the case of standard heats (enthalpy) of formation and normal boiling points of octane isomers, and they concluded that the $AZI$ and $RRR$ index has the best and second-best (respectively) correlating ability among the examined topological indices (it is worth mentioning here that, among the examined topological indices $ABC$ was also included which was the second-best degree based topological index according to the earlier study (Gutman \& To\v{s}ovi\'{c}, 2013)). Hence it is meaningful to study the mathematical properties of the $RRR$ index, especially bounds and characterization of the extremal elements for different graph classes. In (Gutman \textit{et al.}, 2014), the structure of $n$-vertex tree having maximum $RRR$ index and the extremal $n$-vertex graphs with respect to $RRR$ index were reported.

An $n$-vertex (connected) graph $G$ is unicyclic if it has $n$ edges. Some extremal results for the unicyclic graphs can be found in the papers (Gan \textit{et al.}, 2011; Gao \& Lu, 2005; Pan \textit{et al.}, 2006; Zhan \textit{et al.}, 2015). The main purpose of the present note is to characterize the $n$-vertex unicyclic graph having minimum $RRR$ index over the collection of all $n$-vertex unicyclic graphs.

\section*{Main Result}

Denote by $S_{n}^{+}$ the unique unicyclic graph obtained from the star graph $S_{n}$ by adding an edge between any two pendent vertices.
Many topological indices (e.g. $ABC$ index, $R$ index, $AZI$ etc.) which have $S_{n}$ as an extremal graph over the set of all $n$-vertex trees, have also $S_{n}^{+}$ as an extremal graph over the set of all $n$-vertex unicyclic graphs. However, different approaches required to prove these results. From the definition of $RRR$ index, it can be easily seen that $RRR(T_{n})\geq RRR(S_{n})=0$ where $T_{n}$ is any $n$-vertex tree. Is it true that the graph $S_{n}^{+}$ has the minimum $RRR$ index over the set of all $n$-vertex unicyclic graphs? The answer is not positive; for the $n$-vertex unicyclic graph $H_{n}^{+}$ (depicted in Fig. \ref{f1}(b)), one have $RRR(H_{n}^{+})=1+\sqrt{2}(2+\sqrt{n-4})$ but on the other hand $RRR(S_{n}^{+})=1+2\sqrt{n-2}$ and
\[RRR(S_{n}^{+})
\begin{cases}
       <RRR(H_{n}^{+}) & \text{if \ $5\leq n\leq19$,}\\
       =RRR(H_{n}^{+}) & \text{if \ $n=20$,}\\
       >RRR(H_{n}^{+}) & \text{if \ $n\geq21$.}
\end{cases}
\]
\renewcommand{\figurename}{Fig.}
\begin{figure}[H]
   \centering
    \includegraphics[width=3.9in, height=0.9in]{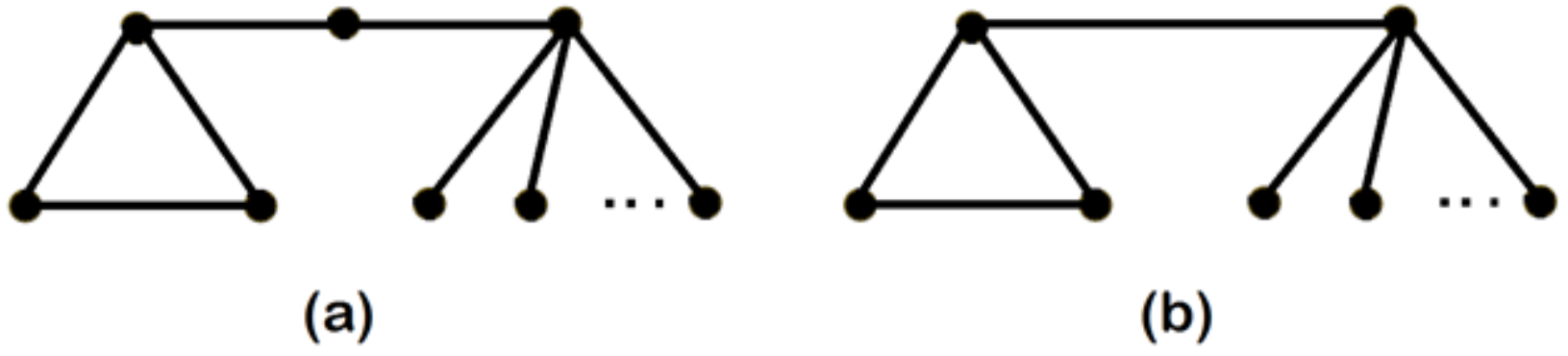}
    \caption{(a) The $n$-vertex unicyclic graph $H_{n}$ where $n$ is at least 6. (b) The $n$-vertex unicyclic graph $H_{n}^{+}$ where $n$ is at least 5.}
    \label{f1}
     \end{figure}

In the following theorem we characterize the $n$-vertex unicyclic graph having minimum $RRR$ index over the collection of all $n$-vertex unicyclic graphs for $n\geq4$.
\begin{thm}\label{thm1}
For any unicyclic graph $U_{n}$ where $n\geq4$, the following inequalities hold:
\[RRR(U_{n})
\begin{cases}
       \geq 1+2\sqrt{n-2} & \text{if \ $4\leq n\leq16$,}\\
       \geq 1+3\sqrt{2}+\sqrt{n-5} & \text{if \ $n\geq17$.}
\end{cases}
\]
The first equality holds if and only if $U_{n}\cong S_{n}^{+}$ and the second equality holds if and only if $U_{n}\cong H_{n}$ (where $H_{n}$ is shown in Fig. \ref{f1}(a)).
\end{thm}

\begin{proof}
Routine computation yields $RRR(C_{n})=n>1+2\sqrt{n-2}$ for all $n\geq4$ and $RRR(C_{n})=n>1+3\sqrt{2}+\sqrt{n-5}$ for all $n\geq7$, so we assume $U_{n}\not\cong C_{n}$. Let $P(U_{n})=\{u_{0},u_{1},u_{2},...,u_{p-1}\}$ be the set of all pendent vertices in $U_{n}$. For $1\leq i\leq p-1$, suppose that $v_{i}$ is adjacent with $u_{i}$ and $W_{u_{i}}$ is the set of all pendent neighbors of $v_{i}$ different from $u_{i}$. Choose a member of $P(U_{n})$, say $u_{0}$ (without loss of generality), such that

\begin{enumerate}
  \item the number of elements in $W_{u_{0}}$ is as large as possible;
  \item subject to (1), $d_{v_{0}}$ is as small as possible.
\end{enumerate}
Let $d_{v_{0}}=x$ and $N(v_{0})=\{u_{0},u_{1},u_{2},...,u_{p-1},u_{p},...,u_{x-1}\}$ where $d_{u_{i}}=1$ for $0\leq i\leq p-1$ and $d_{u_{i}}\geq2$ for $p\leq i\leq x-1$ (see Fig. \ref{f1.1}).
\renewcommand{\figurename}{Fig.}
\begin{figure}[H]
   \centering
    \includegraphics[width=1.9in, height=2.0in]{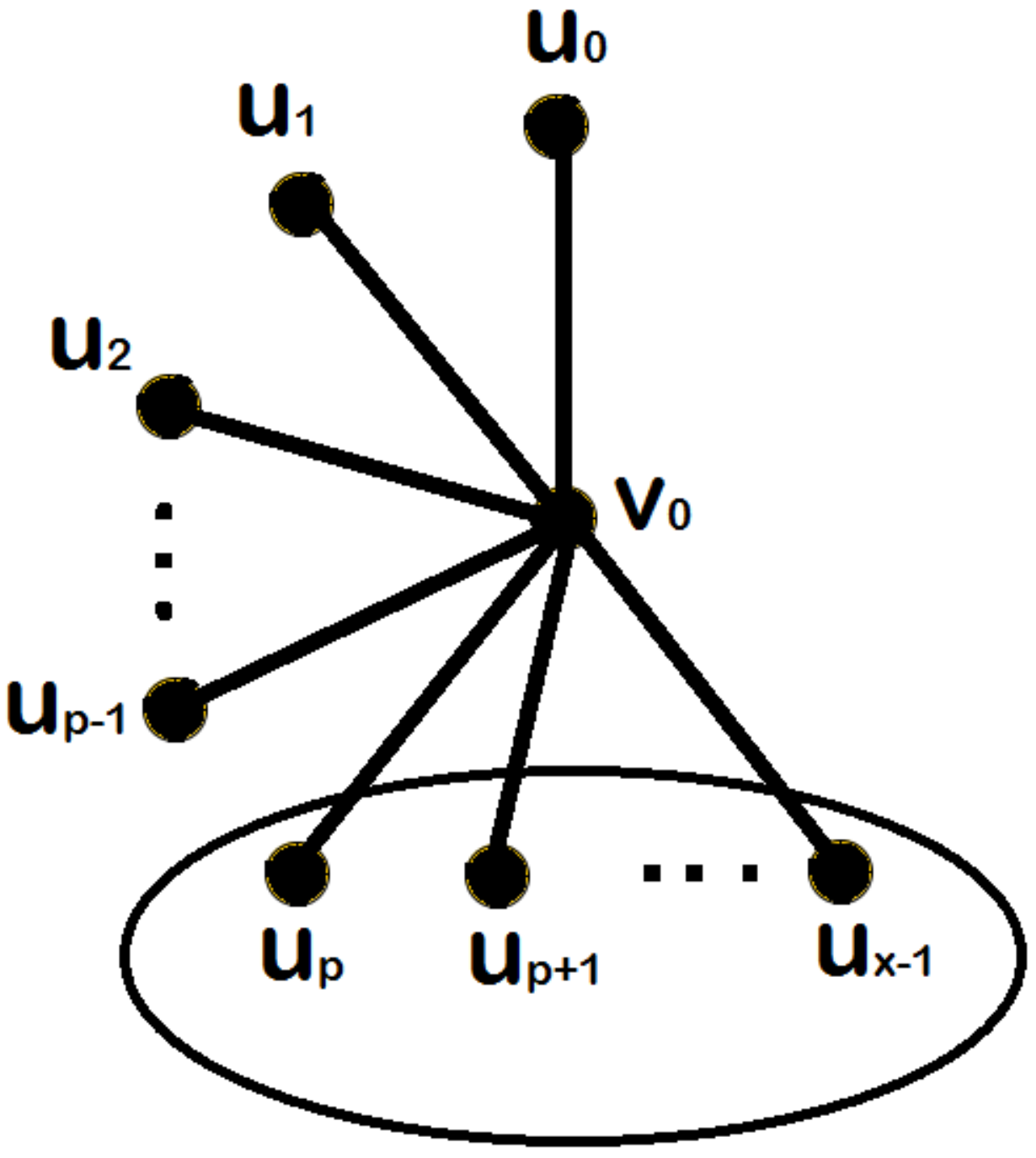}
    \caption{The presentation of an $n$-vertex unicyclic graph $U_{n}$ used in the proof of Theorem \ref{thm1}.}
    \label{f1.1}
     \end{figure}
If $U_{n-1}'$ is the graph obtained from $U_{n}$ by removing the vertex $u_{0}$, then
\begin{equation}\label{Eq.0}
RRR(U_{n})=RRR(U_{n-1}')+\displaystyle\sum_{i=1}^{x-1}\left[\sqrt{(x-1)(d_{u_{i}}-1)}-\sqrt{(x-2)(d_{u_{i}}-1)}\right]
\end{equation}
We will discuss three cases:

\textit{Case 1.} Either the vertex $v_{0}$ is adjacent with at least two non-pendent vertices or $v_{0}$ is adjacent with exactly one non-pendent vertex $u_{x-1}$ such that $d_{u_{x-1}}\geq5$ (that is either $p\leq x-2$ or $p=x-1,d_{u_{x-1}}\geq5$).\\
Let $\mathcal{U}^{(1)}_{n}$ be the collection of all those $n$-vertex unicyclic graphs (different from $C_{n}$) which fall in this case. By using induction on $n$, we will prove that the only one graph, namely $S_{n}^{+}$, has the minimum $RRR$ value among all the members of $\mathcal{U}^{(1)}_{n}$. [Then the desired result will follow from the fact that
\[RRR(S_{n}^{+})=1+2\sqrt{n-2}> 1+3\sqrt{2}+\sqrt{n-5}=RRR(H_{n}) \ \ \ \text{for all $n\geq17$}.]\]
For $n=4$, there are only two non-isomorphic unicyclic graphs namely $C_{n}$ and $S_{n}^{+}$ and hence the result holds for $n=4$. For $n=5$, all the non-isomorphic members of $\mathcal{U}^{(1)}_{n}$ are depicted in the Fig. \ref{f2} along with their $RRR$ values.
\renewcommand{\figurename}{Fig.}
\begin{figure}[H]
   \centering
    \includegraphics[width=4.6in, height=1.0in]{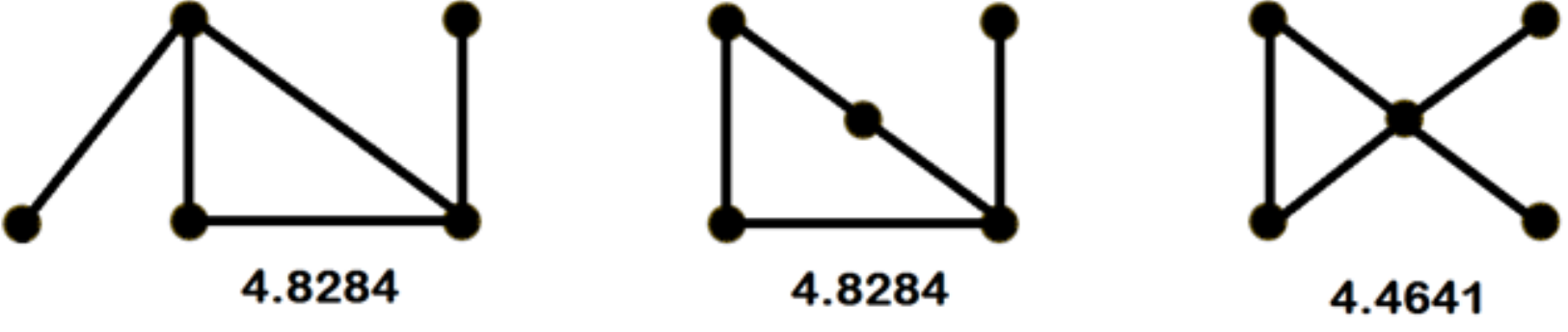}
    \caption{All the non-isomorphic members of $\mathcal{U}^{(1)}_{5}$ together with their $RRR$ values.}
    \label{f2}
     \end{figure}
Now, suppose that $U_{n}\in\mathcal{U}^{(1)}_{n}$ and $n\geq6$. By virtue of inductive hypothesis and from Equation (\ref{Eq.0}), one have
\begin{equation}\label{Eq.1}
RRR(U_{n})\geq 1+2\sqrt{n-3}+\left(\sqrt{x-1}-\sqrt{x-2}\right)\displaystyle\sum_{i=1}^{x-1}\sqrt{d_{u_{i}}-1},
\end{equation}
with equality if and only if $U_{n-1}'\cong S_{n-1}^{+}$.
We discuss two subcases:

\textit{Subcase 1.1.} If $p\geq2$. Then from Inequality (\ref{Eq.1}) it follows that
\begin{equation}\label{Eq.2}
RRR(U_{n})\geq 1+2\sqrt{n-3}+\left(\sqrt{x-1}-\sqrt{x-2}\right)\displaystyle\sum_{i=p}^{x-1}\sqrt{d_{u_{i}}-1}
\end{equation}
According to the definition of $U_{n}\in\mathcal{U}^{(1)}_{n}$, either $p\leq x-2$ or $p=x-1,d_{u_{x-1}}\geq5$. If $p\leq x-2$, then Inequality (\ref{Eq.2}) implies that
\begin{eqnarray*}
RRR(U_{n})&\geq& 1+2\sqrt{n-3}+\left(\sqrt{x-1}-\sqrt{x-2} \ \right)(x-p)\\
&\geq&1+2\sqrt{n-3}+2\left(\sqrt{x-1}-\sqrt{x-2} \ \right)\\
&\geq& 1+2\sqrt{n-3}+2\left(\sqrt{n-2}-\sqrt{n-3} \ \right)=1+2\sqrt{n-2}.
\end{eqnarray*}
The equality $RRR(U_{n})=1+2\sqrt{n-2}$ holds if and only if $x=n-1,x-p=2$ and $U_{n-1}'\cong S_{n-1}^{+}$.

If $p=x-1$ and $d_{u_{x-1}}\geq5$, then the graph $U_{n-1}'$ must be different from $S_{n-1}^{+}$ and hence from Equation (\ref{Eq.0}), one have
\begin{eqnarray*}
RRR(U_{n})&>&1+2\sqrt{n-3}+2\left(\sqrt{x-1}-\sqrt{x-2} \ \right)\\
&>& 1+2\sqrt{n-3}+2\left(\sqrt{n-3}-\sqrt{n-4} \ \right) \ \ (\text{since in this case $x<n-2$})\\
&>&1+2\sqrt{n-2}.
\end{eqnarray*}

\textit{Subcase 1.2.} If $p=1$. Then, from the definition of $u_{0}$, it follows that the set $W_{u_{i}}$ is empty for all $u_{i}\in P(U_{n})$. It means that no two pendent edges are adjacent.

If $x\geq4$, then among the vertices $u_{1},u_{2},...,u_{x-1}$ at least two are disconnected in $U_{n}-v_{0}$ (because otherwise $U_{n}$ contains more than one cycle, a contradiction. See Fig. \ref{f2.1} for the graphs $U_{n}-v_{0}$ and $U_{n}$ considered in this subcase).
\renewcommand{\figurename}{Fig.}
\begin{figure}[H]
   \centering
    \includegraphics[width=3.9in, height=2.0in]{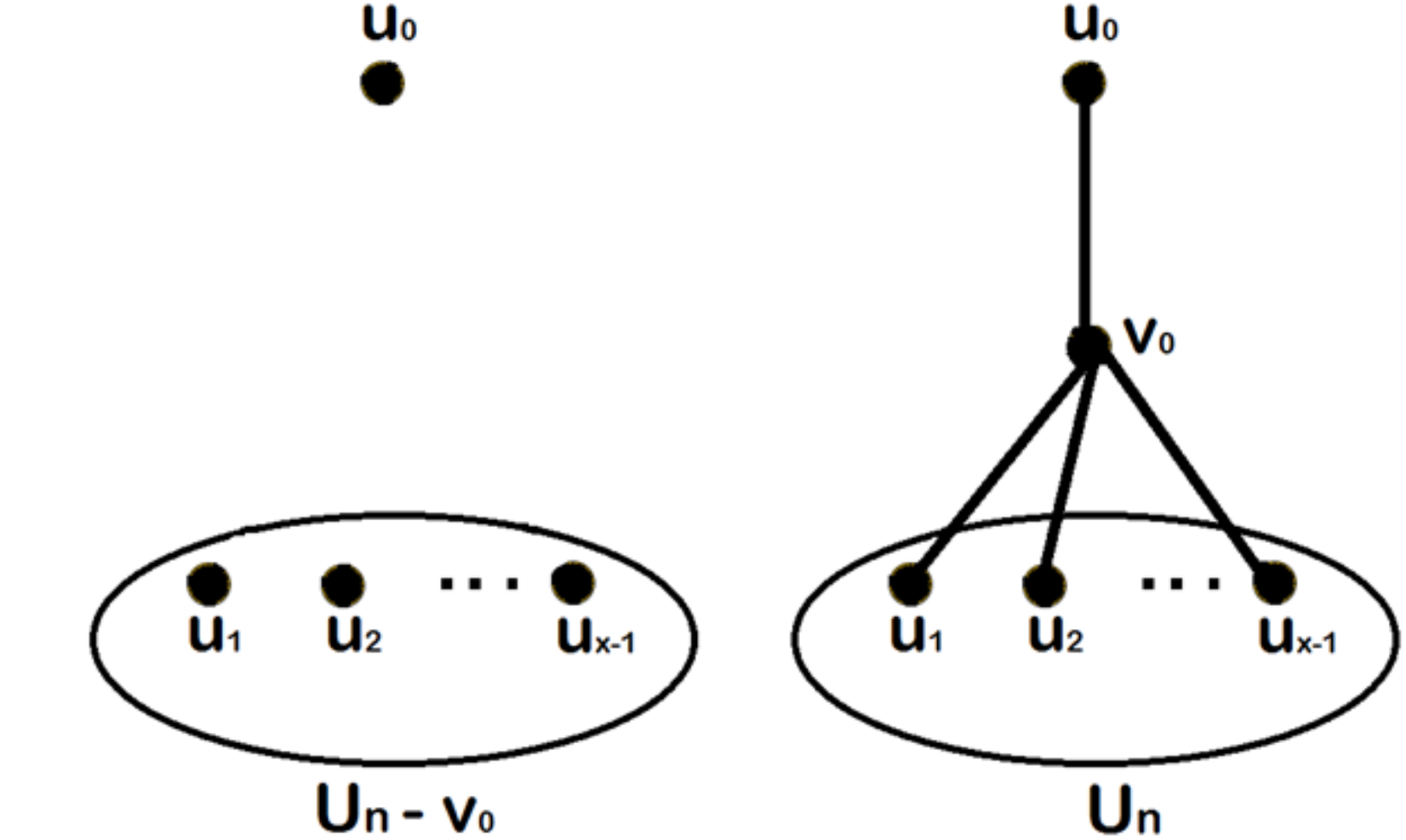}
    \caption{The graphs $U_{n}$ and $U_{n}-v_{0}$ used in Subcase 1.2 of the proof of Theorem \ref{thm1}.}
    \label{f2.1}
     \end{figure}
Without loss of generality, let $u_{1}$ and $u_{2}$ are disconnected in $U_{n}-v_{0}$ and suppose that $C_{i}$ (for $i=1,2$) is the component of $U_{n}-v_{0}$ containing $u_{i}$ (for $i=1,2$). Since $d_{u_{i}}\geq2$ in $U_{n}$ for all $i$ ($1\leq i\leq x-1$), so both the components $C_{1}$ and $C_{2}$ must be non-trivial. Note that at least one of $C_{1}$ and $C_{2}$ must be a tree (for otherwise $U_{n}$ contains more than one cycle, a contradiction). Let $C_{1}$ is a tree. Since every non-trivial tree contains at least two pendent vertices and no two pendent edges of $U_{n}$ are adjacent, so there exist $w_{1}\in V(C_{1})\cap P(U_{n})$ and $w_{1}w_{2}\in V(C_{1})\cap E(U_{n})$ such that $d_{w_{2}}=2$ in $U_{n}$, which contradicts the definition of $u_{0}$. Hence $x=2$ or 3. It should be noted that the graph $U_{n-1}'$ is different from $S_{n-1}^{+}$ in this case. Now, we consider further two subcases:

\textit{Subcase 1.2.1.} If $x=2$. Then from the Inequality (\ref{Eq.0}), we have
\begin{eqnarray*}
RRR(U_{n})&>& 1+2\sqrt{n-3}+\sqrt{d_{u_{1}}-1}\\
&\geq&3+2\sqrt{n-3}>1+2\sqrt{n-2}.
\end{eqnarray*}

\textit{Subcase 1.2.2.} If $x=3$. Then from the Inequality (\ref{Eq.0}), it follows that
\begin{eqnarray*}
RRR(U_{n})&>& 1+2\sqrt{n-3}+\left(\sqrt{2}-1\right)\left(\sqrt{d_{u_{1}}-1}+\sqrt{d_{u_{2}}-1} \ \right)\\
&\geq&1+2\sqrt{n-3}+2\left(\sqrt{2}-1\right)\\
&>&1+2\sqrt{n-2}, \ \ \text{because $n\geq6$}.
\end{eqnarray*}
Therefore, for any $U_{n}\in\mathcal{U}^{(1)}_{n}$ we have $RRR(U_{n})\geq RRR(S_{n}^{+})$ with equality if and only if $U_{n}\cong S_{n}^{+}$.

\textit{Case 2.} The vertex $v_{0}$ is adjacent with exactly one non-pendent vertex $u_{x-1}$ such that $d_{u_{x-1}}=3$ or 4 (that is $p=x-1$ and $d_{u_{x-1}}=3$ or 4).\\
Let $\mathcal{U}^{(2)}_{n}$ be the family of all those $n$-vertex unicyclic graphs (different from $C_{n}$) which fall in this case. Note that $n$ must be at least five in this case. By using induction on $n$, we will prove that the only one graph, namely $H_{n}^{+}$, has the minimum $RRR$ value among all the members of $\mathcal{U}^{(2)}_{n}$. [Then the desired result will follow from the fact that
\[RRR(H_{n}^{+})=1+2\sqrt{2}+\sqrt{2(n-4)}>
\begin{cases}
        1+2\sqrt{n-2} & \text{for $5\leq n\leq16$,}\\
        1+3\sqrt{2}+\sqrt{n-5} & \text{for $n\geq17$.}]
\end{cases}
\]
It can be easily seen that $\mathcal{U}^{(2)}_{5}$ has only one element namely $H_{5}^{+}$. For $n=6$, all the non-isomorphic members of $\mathcal{U}^{(2)}_{n}$ are depicted in the Fig. \ref{f3} along with their $RRR$ values.
\renewcommand{\figurename}{Fig.}
\begin{figure}[H]
   \centering
    \includegraphics[width=3.5in, height=1.92in]{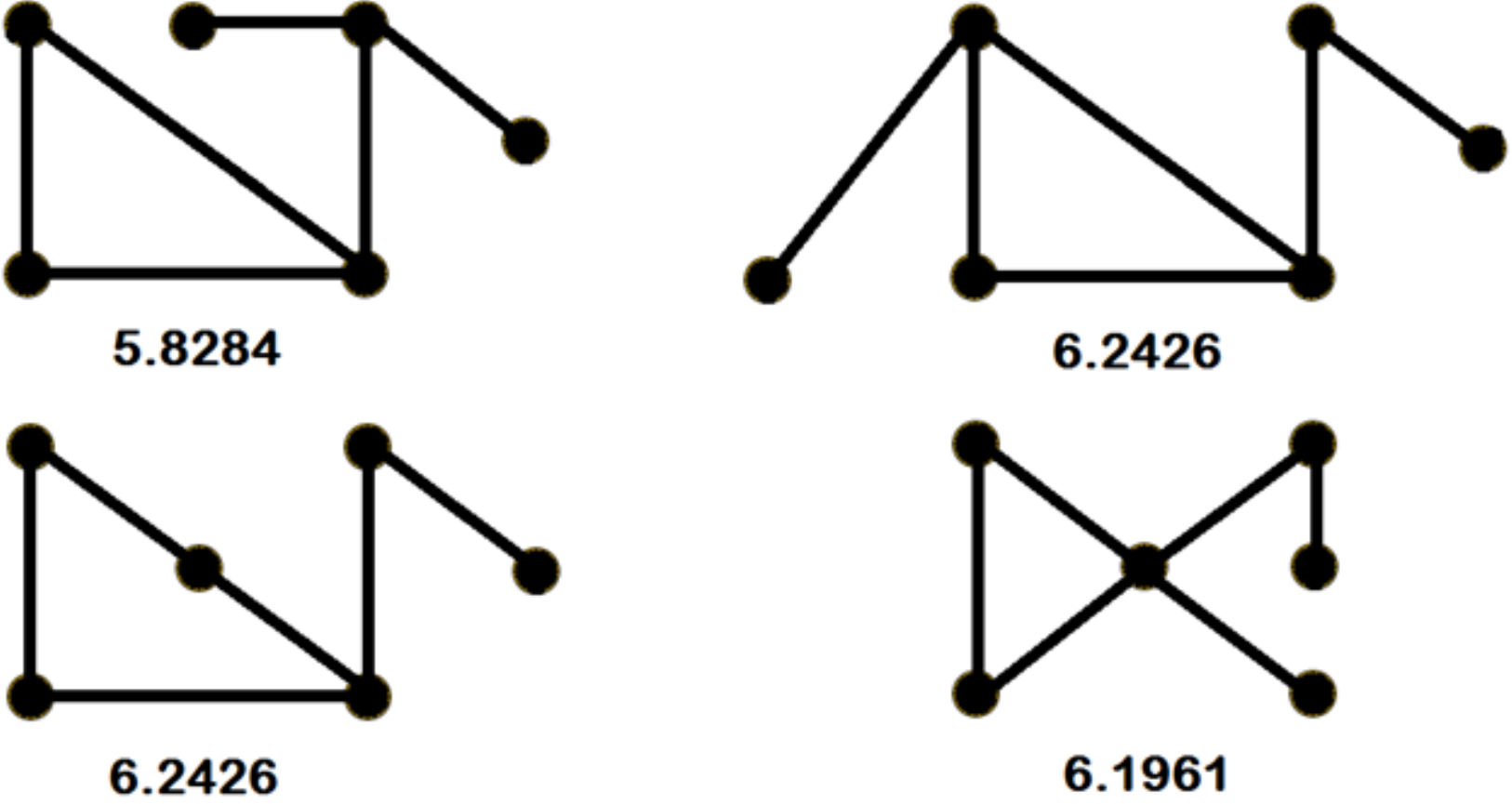}
    \caption{All the non-isomorphic members of $\mathcal{U}^{(2)}_{6}$ together with their $RRR$ values.}
    \label{f3}
     \end{figure}
Hence the result holds for $n=5,6$. Suppose that $U_{n}\in\mathcal{U}^{(2)}_{n}$ and $n\geq7$. By using the inductive hypothesis in the Equation (\ref{Eq.0}), we have
\begin{equation}\label{Eq.11}
RRR(U_{n})\geq 1+2\sqrt{2}+\sqrt{2(n-5)}+\left(\sqrt{x-1}-\sqrt{x-2}\right)\displaystyle\sum_{i=1}^{x-1}\sqrt{d_{u_{i}}-1},
\end{equation}
with equality if and only if $U_{n-1}'\cong H_{n-1}^{+}$.
We consider two subcases:

\textit{Subcase 2.1.} If $p\geq2$. Then from Inequality (\ref{Eq.11}) it follows that
\begin{equation}\label{Eq.22}
RRR(U_{n})\geq 1+2\sqrt{2}+\sqrt{2(n-5)}+\left(\sqrt{x-1}-\sqrt{x-2}\right)\displaystyle\sum_{i=p}^{x-1}\sqrt{d_{u_{i}}-1}
\end{equation}
According to the definition of $U_{n}\in\mathcal{U}^{(2)}_{n}$, $p=x-1$ and $d_{u_{x-1}}=3$ or 4. Hence from Equation (\ref{Eq.22}), it follows that
\begin{eqnarray*}
RRR(U_{n})&\geq&1+2\sqrt{2}+\sqrt{2(n-5)}+\left(\sqrt{x-1}-\sqrt{x-2} \ \right)\sqrt{d_{u_{x-1}}-1}\\
&\geq& 1+2\sqrt{2}+\sqrt{2(n-5)}+\sqrt{2}\left(\sqrt{x-1}-\sqrt{x-2} \ \right)\\
&\geq& 1+2\sqrt{2}+\sqrt{2(n-5)}+\sqrt{2}\left(\sqrt{n-4}-\sqrt{n-5} \ \right) \ (\text{since $x\leq n-3$})\\
&=&1+2\sqrt{2}+\sqrt{2(n-4)}.
\end{eqnarray*}
The equality $RRR(U_{n})=1+2\sqrt{2}+\sqrt{2(n-4)}$ holds if and only if $x=n-3,d_{u_{x-1}}=3$ and $U_{n-1}'\cong H_{n-1}^{+}$.

\textit{Subcase 2.2.} If $p=1$. Then, $x=2$ because $p=x-1$. It is easy to see that the graph $U_{n-1}'$ is different from $H_{n-1}^{+}$ in this case. From the Inequality (\ref{Eq.0}), we have
\begin{eqnarray*}
RRR(U_{n})&>& 1+2\sqrt{2}+\sqrt{2(n-5)}+\sqrt{d_{u_{1}}-1}\\
&\geq&1+3\sqrt{2}+\sqrt{2(n-5)}>1+2\sqrt{2}+\sqrt{2(n-4)}.
\end{eqnarray*}
Therefore, for any $U_{n}\in\mathcal{U}^{(2)}_{n}$ we have $RRR(U_{n})\geq RRR(H_{n}^{+})$ with equality if and only if $U_{n}\cong H_{n}^{+}$.

\textit{Case 3.} The vertex $v_{0}$ is adjacent with exactly one non-pendent vertex $u_{x-1}$ such that $d_{u_{x-1}}=2$ (that is $p=x-1$ and $d_{u_{x-1}}=2$).\\
Let $\mathcal{U}^{(3)}_{n}$ be the class of all those $n$-vertex unicyclic graphs (different from $C_{n}$) which fall in this case. Note that $n$ must be at least 6 in this case. By using induction on $n$, we will prove that the only one graph, namely $H_{n}$, has the minimum $RRR$ value among all the members of $\mathcal{U}^{(3)}_{n}$. [Then the desired result will follow from the fact that
\[RRR(H_{n})=1+3\sqrt{2}+\sqrt{n-5}>1+2\sqrt{n-2} \ \text{ \ for \ \ } 6\leq n\leq16.]\]
It can be easily seen that $\mathcal{U}^{(3)}_{6}$ has only one member, namely $H_{6}$. For $n=7$, all the non-isomorphic members of $\mathcal{U}^{(3)}_{n}$ are depicted in the Fig.\ref{f4} along with their $RRR$ values.
\renewcommand{\figurename}{Fig.}
\begin{figure}[H]
   \centering
    \includegraphics[width=4.8in, height=1.8in]{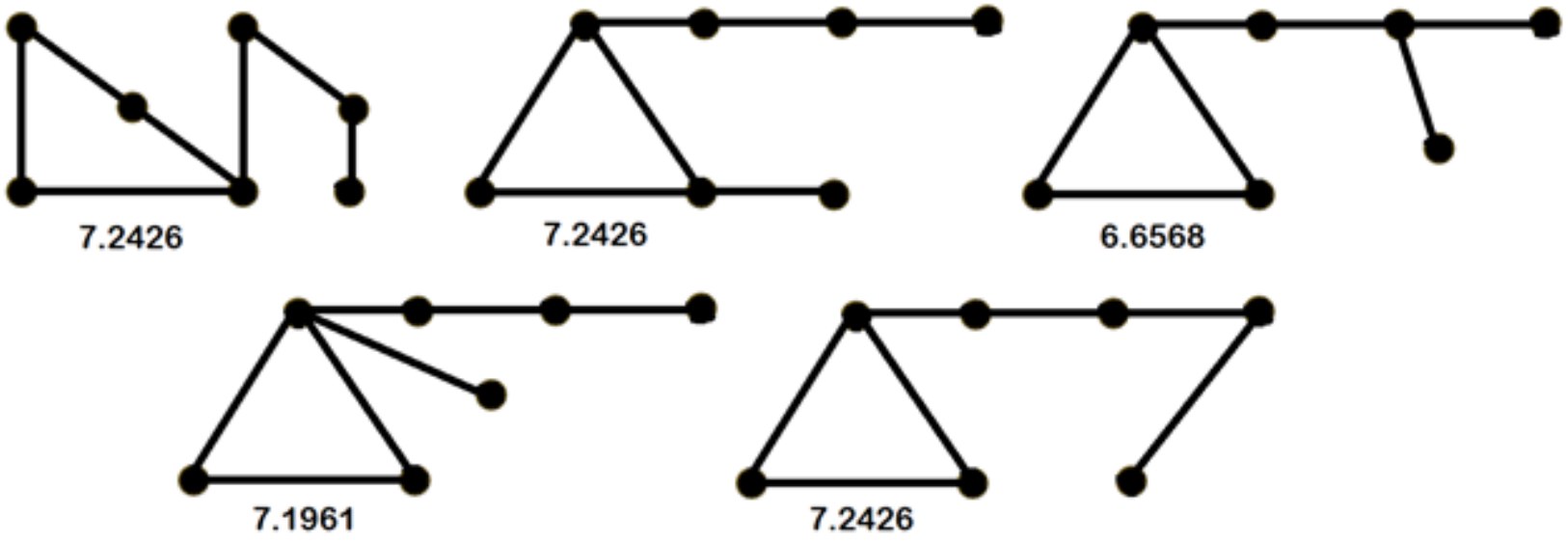}
    \caption{All the non-isomorphic members of $\mathcal{U}^{(3)}_{7}$ together with their $RRR$ values.}
    \label{f4}
     \end{figure}
Hence the result holds for $n=6,7$. Suppose that $U_{n}\in\mathcal{U}^{(3)}_{n}$ and $n\geq8$. By virtue of inductive hypothesis and from Equation (\ref{Eq.0}), one have
\begin{equation}\label{Eq.111}
RRR(U_{n})\geq 1+3\sqrt{2}+\sqrt{n-6}+\left(\sqrt{x-1}-\sqrt{x-2}\right)\displaystyle\sum_{i=1}^{x-1}\sqrt{d_{u_{i}}-1},
\end{equation}
with equality if and only if $U_{n-1}'\cong H_{n-1}$.
We consider two subcases:

\textit{Subcase 3.1.} If $p\geq2$. Then from Inequality (\ref{Eq.111}) it follows that
\begin{equation}\label{Eq.222}
RRR(U_{n})\geq 1+3\sqrt{2}+\sqrt{n-6}+\left(\sqrt{x-1}-\sqrt{x-2}\right)\displaystyle\sum_{i=p}^{x-1}\sqrt{d_{u_{i}}-1}
\end{equation}
By the definition of $U_{n}\in\mathcal{U}^{(3)}_{n}$, $p=x-1$ and $d_{u_{x-1}}=2$. Hence from Equation (\ref{Eq.222}), it follows that
\begin{eqnarray*}
RRR(U_{n})&\geq&1+3\sqrt{2}+\sqrt{n-6}+\sqrt{x-1}-\sqrt{x-2}\\
&\geq& 1+3\sqrt{2}+\sqrt{n-5}.
\end{eqnarray*}
The last inequality holds because $x\leq n-4$. Furthermore, the equality $RRR(U_{n})=1+3\sqrt{2}+\sqrt{n-5}$ holds if and only if $x=n-4$ and $U_{n-1}'\cong H_{n-1}$.

\textit{Subcase 3.2.} If $p=1$. Then, $x=2$ because $p=x-1$. Observe that the graph $U_{n-1}'$ is different from $H_{n-1}$ in this case. From the Inequality (\ref{Eq.0}), we have
\begin{eqnarray*}
RRR(U_{n})&>& 1+3\sqrt{2}+\sqrt{n-6}+\sqrt{d_{u_{1}}-1}\\
&=&2+3\sqrt{2}+\sqrt{n-6}>1+3\sqrt{2}+\sqrt{n-5}.
\end{eqnarray*}
Therefore, for any $U_{n}\in\mathcal{U}^{(3)}_{n}$ we conclude that $RRR(U_{n})\geq RRR(H_{n})$ with equality if and only if $U_{n}\cong H_{n}$. This completes the proof.

\end{proof}

\end{document}